\documentclass[a4paper]{article}

\usepackage[a4paper]{geometry}
\usepackage[T1]{fontenc}
\usepackage[utf8]{inputenc}
\usepackage{amsmath}
\usepackage{amssymb}
\usepackage{enumitem}
\usepackage{graphicx}

\usepackage{amsthm}
\usepackage[unicode, colorlinks=true, linkcolor=blue, urlcolor=magenta, citecolor=blue]{hyperref}
\usepackage[capitalize, nameinlink]{cleveref}

\newtheorem{lemma}{Lemma}
\newtheorem{proposition}[lemma]{Proposition}
\newtheorem{theorem}[lemma]{Theorem}

\newtheorem*{mainthm}{Main theorem}
\newtheorem{cor}[lemma]{Corollary}
\theoremstyle{definition}
\newtheorem{definition}[lemma]{Definition}
\newtheorem{remark}[lemma]{Remark}

\usepackage[backend=biber,
            backref=true,
            isbn=false,
            doi=false,
            sortcites=false,
            maxbibnames=99,
            style=alphabetic]{biblatex}
\DeclareMathOperator{\Id}{Id}
\DeclareMathOperator{\R}{\mathcal{R}}
\DeclareMathOperator{\Hom}{Hom}
\DeclareMathOperator{\End}{End}
\newcommand{\co}{\!:}
\DeclareMathOperator{\Rel}{\mathcal{R}}
\DeclareMathOperator{\Span}{Span}
\newcommand\rst[2]{{{#1}\vert_{{#2}}}}

\title{Holonomic approximation through convex integration}
\author{Patrick Massot\thanks{The first author was partially funded by ANR
grant Microlocal ANR-15-CE40-0007} \and Mélanie Theillière\thanks{The second author was supported by ANR/FNR project SoS, INTER/ANR/16/11554412/SoS, ANR-17-CE40-0033.}}

\begin{document}
\maketitle

\begin{abstract}
  Convex integration and the holonomic approximation theorem are two well-known
  pillars of flexibility in differential topology and geometry.
  They may each seem to have their own flavor and scope.
  The goal of this paper is to bring some new perspective on this topic.
  We explain how to prove the holonomic approximation theorem for first order jets using
  convex integration. More precisely we first prove that this theorem can easily be
  reduced to proving flexibility of some specific relation. Then we prove
  this relation is open and ample, hence its flexibility follows from
  off-the-shelf convex integration.
\end{abstract}

\section*{Introduction}

\subsection*{The $h$-principle techniques landscape}

Since Gromov and Eliashberg seminal work in
\cite{Gromov_ICM,Gromov_Eliashberg_elim}, we know it is often fruitful to ask
which geometrical construction problems satisfy the $h$-principle.
This principle is too broad to be fully described abstractly, but a lot of
examples can be described using the formalism of jet-spaces. In this paper,
it will be enough to consider the space $J^1(M, N)$ of $1$-jets of maps between
two manifolds $M$ and $N$.
This is a bundle over $M \times N$ whose fiber at $(m, n)$ is the vector space of
linear maps from $T_mM$ to $T_nN$.
Slightly bending notations, elements of $J^1(M, N)$ are usually written as
triples $(m, n, \varphi)$, with $\varphi \in \Hom(T_mM, T_nN)$.
Composing with the projection $M \times N \to M$, this jet space also fibers over $M$.
Any smooth map $f : M \to N$ gives rise to the section $j^1f : M \to J^1(M, N)$
sending $m$ to $(m, f(m), T_mf)$, abbreviated as $j^1f = (f, Tf)$.
Sections of this form are called holonomic sections.
A first order partial differential relation for maps from $M$ to $N$
is simply a subset $\Rel$ of $J^1(M, N)$.
A formal solution of $\Rel$ is a section of $J^1(M, N)$ taking values in
$\Rel$. Following Gromov and Eliashberg, we say $\Rel$ satisfies
the $h$-principle if every formal solution is homotopic to a holonomic one.

Understanding whether a given $\Rel$ satisfies the $h$-principle is usually
difficult. One reason for this is the wide range of techniques to try.
Gromov's partial differential relations book
\cite{Gromov_PDR} and Eliashberg and Mishachev's book
\cite{Eliashberg_Mishachev} explain the following techniques: removal of
singularities, inversion of differential operators, convex
integration and holonomic approximation. The latter two are the most
important for differential topology (including symplectic and contact
topology). The holonomic approximation theorem, for first order relations,
immediately implies for instance the $h$-principle for immersions in positive
codimension, directed embeddings of open manifolds, existence and deformations
of symplectic or contact structures on open manifolds. On the other hand convex
integration, rather directly imply the Nash-Kuiper isometric embedding theorem,
immediately implies flexibility of important classes of immersions and inspired
important work in PDEs.

Holonomic approximation and convex integration apply to problems that are fully
flexible but they are also crucial first steps to study subtler problems. One
outstanding example is Murphy's $h$-principle for loose Legendrian embeddings in
\cite{Murphy_loose}. That paper uses the holonomic approximation theorem, convex
integration and wrinkling techniques before a specific argument using the
geometric looseness assumption.

All those examples seems to indicate that holonomic approximation and convex
integration have distinct flavors. The goal of this paper is to bring some more
unity to this topic by proving the holonomic approximation theorem for first
order jets using convex integration. We first prove that it can easily
be reduced to proving the $h$-principle for some specific relation. Then we
prove this relation is open and ample, hence solvable using convex integration
as a black-box.

We do \emph{not} claim that this proof of the holonomic approximation theorem is
\emph{better} than the existing proof (whatever it could mean), or reveals the
deep nature of this theorem. But we think it brings some extra insight.

The first author was also motivated by his interest in formalized mathematics.
Explaining non-trivial mathematics to a computer is still much more challenging
than convincing a human being, so deducing one difficult theorem from another is
a great gain. Also, the traditional proof of the holonomic approximation theorem
is rather monolithic and requires great care to get all details and
constructions right at the same time. By contrast, convex integration,
especially the implementation in \cite{these_melanie}, is built out of a series
of very cleanly encapsulated steps. So there is hope this version of the story
will be easier to formalize.  Convex integration is also more explicit so this
new proof may be relevant from the point of view of effective $h$-principles and
visualization (we will come back to this topic at the end of this introduction).

\subsection*{Holonomic approximation}

Holonomic approximations cannot exist globally on a manifold, they exist near
polyhedra with positive codimension. Because it claims a $C^0$-close and relative
construction, it is a completely local statement. The global version is proved
by a straightforward induction on cells of polyhedra. We do not change anything
here. By contrast, the induction on directions on each cell does disappear in our
proof, it becomes part of the black-boxed convex integration.

We set $A = [0, 1]^m \times \{0\} \times \{0\}\subset \mathbb{R}^m \times \mathbb{R} \times \mathbb{R}^k$ which is a model for
cells with positive codimension. We split the normal direction as
$\mathbb{R} \times \mathbb{R}^k$ to emphasize we will use only one extra direction. Nothing will happen
in the $\mathbb{R}^k$ direction.
We will always denote by $(x, y, z)$ the coordinates on $\mathbb{R}^m \times \mathbb{R} \times \mathbb{R}^k$.
The target of our maps will be $\mathbb{R}^n$ with coordinate $w$.
We will deform $A$ using functions $\delta : \mathbb{R}^m \to \mathbb{R}$. The deformed $A$ corresponding
to such a function is the graph
\[
  A_\delta = \{(x, y, z) \in \mathbb{R}^m \times \mathbb{R} \times \mathbb{R}^k \;|\; x \in A,\, y = \delta(x) \text{ and } z = 0 \}.
\]

\begin{definition}
Let $\sigma=(f, \varphi)$ be a section of $J^1(\mathbb{R}^m \times \mathbb{R} \times \mathbb{R}^k, \mathbb{R}^n)$ defined near $A$. A
pair $(\delta, f_1)$, with $\delta : \mathbb{R}^m \to \mathbb{R}$ and $f_1$ defined near $A_\delta$ to $\mathbb{R}^n$, is a
\textit{solution of the holonomic $\epsilon$-approximation problem for $\sigma$ near
$A$} if $\|\delta\|_0 < \epsilon$ and $\|j^1 f_1 - \sigma\|_0 < \epsilon$ near $A_\delta$.
\end{definition}

The above statement uses the $C^0$ norm, denoted by $\|\cdot\|_0$, for smooth
functions defined on compact sets. This norm depends on the norm chosen at the
source and target spaces, but those spaces will always be finite dimensional so
we will be free to choose them without changing anything important. In this
article, it will be convenient to endow the target $\mathbb{R}^n$ with the sup norm and
the source $\mathbb{R}^m \times \mathbb{R} \times \mathbb{R}^k$ with the Euclidean norm.

As stated in the first part of the next theorem, and explained in
Section~\ref{sub:the_extension_problem}, the holonomic $\epsilon$-approximation problem,
for any $\sigma$ near $A$, is all about finding the deformation map $\delta$ and a function
$h$ on its graph $A_\delta$ which approximate the induced section of $J^1(A_\delta, \mathbb{R}^n)$. There
is no obstruction to explicitly extend such a function to a full solution.
This observation allows to transform the holonomic approximation problem
for $\sigma = (f, \varphi)$ near $A$, which a priori takes place in
$J^1(\mathbb{R}^m \times \mathbb{R} \times \mathbb{R}^k, \mathbb{R}^n)$, into solving a partial differential relation on a pair
of functions $\delta : \mathbb{R}^m \to \mathbb{R}$ and $h : \mathbb{R}^m \to \mathbb{R}^n$ defined near $[0, 1]^m$,
which takes place in $J^1(\mathbb{R}^m, \mathbb{R} \times \mathbb{R}^n)$.

An element of $J^1(\mathbb{R}^m, \mathbb{R} \times \mathbb{R}^n)$ is a tuple $(x, (y, w), (Y, W))$ where $x$ is in
$\mathbb{R}^m$, $(y, w)$ is in $\mathbb{R} \times \mathbb{R}^n$ and $(Y, W)$ is in $\Hom(\mathbb{R}^m, \mathbb{R} \times \mathbb{R}^n)$. In
particular $Y$ is a linear form on $\mathbb{R}^m$ and its graph in $\mathbb{R}^m \times \mathbb{R}$ will be
denoted by $\Gamma_Y$. Note that, when $Y = d\delta(x)$,
$T_{(x, \delta(x), 0)}A_\delta = \Gamma_Y \times \{0\}$.
We also denote by $p_m$ the projection of $\mathbb{R}^m \times \mathbb{R} \times \mathbb{R}^k$ onto $\mathbb{R}^m$.

\begin{definition}
  \label{def:rel_ha}
  The holonomic approximation relation associated to a section $\sigma = (f, \varphi)$ defined near $A$ and
  a positive real number $\epsilon$ is:
  \[
    \Rel_{ha}(\sigma, \epsilon) = \left\{
     \big(x, (y, w), (Y, W)\big) \in J^1(\mathbb{R}^m, \mathbb{R} \times \mathbb{R}^n) \;\left|\;
    \begin{aligned}
        &|y| < \epsilon \text{, }\; \|w - f(x, y, 0)\| < \epsilon \\
        &\left\|\big(W \circ p_m - \varphi(x, y, 0)\big)_{|\Gamma_Y \times \{0\}}\right\| < \epsilon
    \end{aligned}\right.\right\}
  \]
  where $\Gamma_Y$ is the graph of $Y$ and $p_m$ is the projection of $\mathbb{R}^m \times \mathbb{R} \times \mathbb{R}^k$ onto $\mathbb{R}^m$.
\end{definition}

\begin{mainthm}
  Let $\sigma = (f, \varphi)$ be a section of $J^1(\mathbb{R}^m \times \mathbb{R} \times \mathbb{R}^k , \mathbb{R}^n)$ defined near $A$ and let $\epsilon$ be a positive real number.

  \begin{enumerate}
    \item
    Let $\delta : \mathbb{R}^m \to \mathbb{R}$ and $h : \mathbb{R}^m \to \mathbb{R}^n$ be smooth maps.
    The function $(x, \delta(x), 0) \mapsto h(x)$, defined on $A_\delta$, can be extended to a
    function $f_1$ defined near $A_\delta$ such that $(\delta, f_1)$ is a solution to the
    holonomic $\epsilon$-approximation problem for $\sigma$ near $A$ if and only if $(\delta, h)$
    is a solution of $\Rel_{ha}(\sigma, \epsilon)$.

    \item
      The relation $\Rel_{ha}(\sigma, \epsilon)$ is open and ample.
  \end{enumerate}
\end{mainthm}

The above theorem immediately implies the following local statement which is the
heart of Eliashberg and Mishachev's holonomic approximation theorem for
1-jets (going from this to the global statement is a straightforward induction
on simplices of a polyhedron).

\begin{cor}[{\cite{Eliashberg_holonomic_paper}}]
  \label{thm:EM_intro}
  Let $\sigma$ be a section of $J^1(\mathbb{R}^m \times \mathbb{R} \times \mathbb{R}^k, \mathbb{R}^n)$ defined near $A$. For every $\epsilon > 0$, there exist a solution $(\delta, f_1)$ of the holonomic $\epsilon$-approximation problem for $\sigma$ near $A$. This holds in relative form and parametrically.
\end{cor}

\begin{proof}
  The relation $\Rel_{ha}(\sigma, \epsilon)$ admits a formal solution
  $x \mapsto \big((0, f(x, 0, 0)), (0, \varphi(x, 0, 0)\big)$.
  Since $\Rel_{ha}(\sigma, \epsilon)$ is open and ample, convex integration finishes the proof
  (see Section~\ref{sec:ample_relations_and_convex_integration} for information
  about ampleness and convex integration).
\end{proof}

As a last general remark, note that the holonomic approximation theorem also
holds for higher order jet spaces, although this more general version has
comparatively very few applications.
It seems that usual convex integration methods cannot directly prove
this version because they work one derivation order at a time. In our situation,
one could start with a section $\sigma$ of $J^2(M, N)$ which is holonomic up to order
one and get a relation for $(\delta, h)$ admitting a formal solution which is also
holonomic up to order one with vanishing $\delta$. Convex integration would then try
to produce a solution with very small first order derivative for $\delta$, which is
hopeless. However this does not prevent existence of some variation on the
convex integration idea that could work in that setup.

\subsection*{The mountain path example}

It can be helpful to see how the classical example of mountain paths looks like here. In that case $m = 1$, $k = 0$ and $n = 1$. Since $k = 0$ there is no $z$
coordinate in this case.

Here we consider the section $\sigma$ given by $f : (x, y) \mapsto x$ and $\varphi : (x, y) \mapsto 0$,
in other words we want to walk up on a mountain path with almost zero slope. The
map $\gamma \co x \mapsto (x, \delta(x), h(x))$ then parametrizes the core of the desired
mountain path. In that case $\Gamma_Y \times \{0\}$ is spanned by $\partial_x + Y(\partial_x)\partial_y$
which becomes $\partial_x + \delta'(x)\partial_y$ when $Y = d\delta(x)$.
The constraints on $(\delta, h)$ given by the main theorem are then, for all
$x \in [0, 1]$,

\[
  \begin{cases}
      |\delta(x)| < \epsilon \\
      |h(x) - x| < \epsilon \\
      |h'(x)| < \epsilon\sqrt{1 + \delta'(x)^2}
  \end{cases}
\]

Note the last line has a very clear geometric meaning. Squaring both sides of
the inequality, we see that the derivative $(1, \delta', h')$ of $\gamma$ must be in the
cone defined by $W^2 < \epsilon^2(X^2 + Y^2)$, and this condition is indeed an $\epsilon$-relaxed
version of being in the horizontal plane $W = 0$ specified by $\varphi$.
Here the ampleness condition means that the convex hull
of $\Omega := \{(Y, W) \in \mathbb{R}^2 \,|\, W^2 < \epsilon^2(1+Y^2)\}$ is the whole plane.
This $\Omega$ is the connected component of the complement of a hyperbola which
contains the asymptotes (this hyperbola is the intersection of the cone
$W^2 = \epsilon^2(X^2 + Y^2)$ with the affine plane $X = 1$). So the convex hull assumption
is indeed satisfied.

Specifically, the implementation of convex integration from \cite{these_melanie}
uses, for each $x$, a loop $\gamma : \mathbb{S}^1 \to \Omega$ based at the formal derivative
$(0, \varphi(x, 0)) = (0, 0)$, taking values in $\Omega$ and with average value the derivative
$(0, 1)$ of the zeroth order part of the formal solution: $x \mapsto (0, x)$.
For instance we can choose the loop $t \mapsto \big(4\sin(2\pi t)/\epsilon, 2\sin^2(2\pi t)\big)$.
Convex integration then produces
\[
  \delta : x \mapsto \frac{2(1 - \cos(2\pi N\,x))}{\epsilon\pi N}, \quad
  h : x \mapsto x - \frac1{4\pi N}\sin(4\pi Nx).
\]
where $N$ is a positive number that should be chosen large enough. Looking at
the constraints on $\delta$ and $h$, we deduce that $N \geqslant 4/(\pi\epsilon^2)$ is large
enough (provided $\epsilon < 1$).

Together, those functions parametrize the core of the mountain path, \emph{ie} the graph of the holonomic approximation restricted to the deformed submanifold $A_\delta$. As expected (but not plugged in the construction!), the deformed submanifold oscillates and the value $h$ goes up except at the turning point where $\delta'(x) = 0$ (note how the frequency of $h$ is twice the frequency of $\delta$).

Extending this solution from the deformed submanifold to a neighborhood can also
be done explicitly (see \cref{rem:explicite}), leading to a fully explicit
solution pictured in \cref{fig:mountain} (still using the above explicit expression for $\delta$ and $h$):
\[
  f_1 : (x, y) \mapsto h(x) +
  4\epsilon\,(y - \delta(x))\frac{(1-\cos(4\pi N x))\sin(2\pi Nx)}{\epsilon^2 + 16\sin^2(2\pi Nx)}.
\]

Besides its potential pedagogical value, the full explicitness of this example
shows that our proof is relevant from the point of view of $h$-principle
visualization, as pioneered by \cite{hevea_tore_plat}.

\begin{figure}[!ht]
\centering
\includegraphics[width=0.4\textwidth]{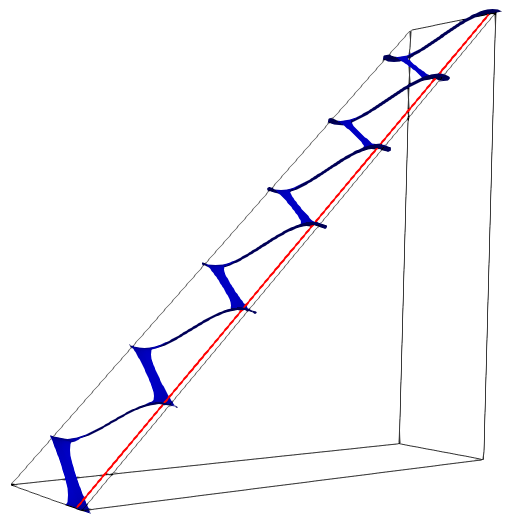}
\caption{The straight line in red is the starting path $f(x,y)=x$ walking up the mountain. The narrow surface in blue is the extended solution $f_1$ for $N=6$ and $\epsilon=1$.}
\label{fig:mountain}
\end{figure}

\paragraph{Outline}

\Cref{sec:ample_relations_and_convex_integration} is a purely expository
section where we recall the definition of ample differential relations and
their very strong $h$-principle.
\Cref{sec:holonomic_approximation} then proves the main theorem. The two parts
of that statement are proved completely independently. First,
Section~\ref{sub:the_extension_problem} explains the easy extension part. Then
Section~\ref{sub:ampleness_for_holonomic_approximation} is the heart of this
paper, proving ampleness of the relation $\Rel_{ha}(\sigma, \epsilon)$
from \cref{def:rel_ha}.

\paragraph{Acknowledgment}

The first author discussed the dream of proving the holonomic approximation
theorem through convex integration with Yakov Eliashberg during Jean Cerf's
90th birthday conference in 2018. The fact that Eliashberg didn't dismissed this
idea as crazy was of course supremely encouraging.
We also thank Vincent Borrelli for useful suggestions about the exposition in
this paper.

\section{Ample relations and convex integration}%
\label{sec:ample_relations_and_convex_integration}

This section recalls the definition of ample relations and Gromov's flexibility
theorem for open and ample differential relations, in particular setting up
notations we will need in the next section.

\begin{definition}[{Gromov in \cite[Section~2.4.C]{Gromov_PDR}, see also \cite[Sections~18.1 and~18.3]{Eliashberg_Mishachev}}]
  \label{def:principal}
  Let $M$ and $N$ be manifolds. For every hyperplane $\tau \subset T_mM$ for some $m$ and every linear map $L \co \tau \to T_nN$ for some $n$, we denote by $P_{\tau, L}$ the principal affine subspace of linear maps $\varphi \co T_mM \to T_nN$ whose restriction to $\tau$ agrees with $L$.
\end{definition}

Each principal subspace $P_{\tau, L}$ naturally lives inside some $\Hom(T_mM, T_nN)$ and can be seen as a subset of $J^1(M, N)$.
	Choosing some $v \in T_mM \setminus \tau$ allows to see $P_{\tau, L}$ more conveniently in $T_nN$. Indeed the map sending $\varphi \in P_{\tau, L}$ to $\varphi(v)$ is then
an affine isomorphism. To describe its inverse, consider the unique
linear form $\pi \in T^*_mM$ such that $\ker \pi = \tau$ and $\pi(v) = 1$, and consider the
projection $p$ of $T_mM$ onto $\tau$ in the decomposition $T_mM = \tau \oplus \Span(v)$.
The inverse of the above isomorphism can be expressed as $w \mapsto L \circ p + w \otimes \pi$.

Let $\Rel$ be a first order partial differential relation
for maps from $M$ to $N$, ie a subset of $J^1(M, N)$.
For every $\tau$ and $L$ as above, we get a slice $\Rel \cap P_{\tau, L}$.

\begin{definition}[Gromov in \cite{Gromov_73}]
  \label{def:ample}
  A subset $\Omega$ of a real affine space is \textit{ample} if the convex
  hull of each connected component of $\Omega$ is the whole space.
  A first order partial differential relation $\Rel$ is ample if,
  for each principal subspace $P_{\tau, L}$, the slice
  $\Rel \cap P_{\tau, L}$ is ample.
\end{definition}

Note that, for every affine space, the empty subset is ample since it has no
connected component.
The above definition is stated in the context of maps between manifolds but,
with a bit more care, it can be extended to the context of sections of bundles
and to higher order jets.

The most famous example of an ample relation is the
immersion relation in positive codimension. The corresponding slices are
complements of linear subspaces having codimension at least 2. Those slices are
obviously ample. In the present paper, ampleness will be less obvious.
Open and ample relations are somehow the most flexible of all partial
differential relations. They satisfy the strongest forms of the $h$-principle
without any condition on the topology of the source or target manifolds.

\begin{theorem}[Gromov, in \cite{Gromov_73, Gromov_PDR}]
  \label{thm:gromov_ample}
  Open and ample partial differential relations of order one satisfy all flavors
  of the $h$-principle: with parameters, relative and $C^0$-close.
\end{theorem}

A first important remark about this theorem is that it is a purely local result.
The general case obviously implies the case where the source and target manifolds
are open sets in affine spaces, since those are manifolds. Conversely,
because the $h$-principle obtained is both relative and $C^0$-close, this special
case implies the general case, working one local chart at a time.

A second remark is that parameters come for free. Say we are interested in
families of maps from $M$ to $N$ parametrized by a manifold $P$.
Denote by $\Psi$ the map from $J^1(M \times P, N)$ to $J^1(M, N)$ sending $(m, p, n, \psi)$ to
$(m, n, \psi \circ \iota_{m, p})$ where $\iota_{m, p} : T_mM \to T_mM \times T_pP$ sends $v$ to $(v, 0)$.
To any family of sections $F_p : m \mapsto (f_p(m), \varphi_{p, m})$ of $J^1(M, N)$, we
associate the section $\bar F$ of $J^1(M \times P, N)$ sending $(m, p)$ to
$\bar F(m, p) := (f_p(m), \varphi_{p, m} \oplus \partial f/\partial p(m, p))$. Then $F$ is a family of
formal solutions of some relation $\Rel \subset J^1(M, N)$ if and only if $\bar F$ is a
formal solution of $\Rel_P := \Psi^{-1}(\Rel)$. In addition $\bar F$ is holonomic
at $(m, p)$ if and only if $F_p$ is holonomic at $m$.
One can check that if $\Rel$ is ample then, for any parameter space $P$,
$\Rel_P$ is also ample. Hence one can completely ignore parameters
when proving \Cref{thm:gromov_ample}, without needing arguments such as
``handling parameters only complicate notations''.

\section{Holonomic approximation}%
\label{sec:holonomic_approximation}

\subsection{The extension problem}%
\label{sub:the_extension_problem}

This section is devoted to the following result which is the easier half of the main theorem.

\begin{proposition}
  \label{prop:extension}
  Let $\sigma = (f, \varphi)$ be a section of $J^1(\mathbb{R}^m \times \mathbb{R} \times \mathbb{R}^k , \mathbb{R}^n)$ defined near $A$ and
  let $\epsilon$ be a positive real number.
  Let $\delta : \mathbb{R}^m \to \mathbb{R}$ and $h : \mathbb{R}^m \to \mathbb{R}^n$ be smooth maps.
  The function $(x, \delta(x), 0) \mapsto h(x)$, defined on $A_\delta$, can be extended to a
  function $f_1$ defined near $A_\delta$ such that $(\delta, f_1)$ is a solution to the
  holonomic $\epsilon$-approximation problem for $\sigma$ near $A$ if and only if $(\delta, h)$
  is a solution of the partial differential relation $\Rel_{ha}(\sigma, \epsilon)$
  from \cref{def:rel_ha}.
\end{proposition}

The first relevant fact is that $T_{(x, \delta(x), 0)}A_\delta = \Gamma_{d\delta} \times \{0\}$.
This explains the appearance of $\Gamma_Y \times \{0\}$ in the relation and allows to
prove that $\Rel_{ha}(\sigma, \epsilon)$ expresses a necessary condition for the existence of
$f_1$.

In order to prove this condition is sufficient, we then need to extend a
function defined on $A_\delta \subset \mathbb{R}^m \times \mathbb{R} \times \mathbb{R}^k$.
The tangent space $H_1$ to the graph of the desired extension $f_1$ at each
$p = (x, y, z, f_1(x, y, z))$ must be close to the affine subspace $H_0$
going through $p$ with direction the graph of linear map $\varphi(x, y, z)$.  In this
extension problem we already know $H_1$ above $TA_\delta$ and the relation $\Rel_{ha}(\sigma, \epsilon)$
ensures this part is close to $H_0$. We then extend linearly on each fiber of the
Euclidean tubular neighborhood of $A_\delta$, making sure that $H_1$ equals $H_0$ above
$TA_\delta^\perp$. The proof below tells this story using more formulas.

\begin{proof}
  By definition, the pair $(\delta, h)$ is a solution of $\Rel_{ha}(\sigma, \epsilon)$ near $A$ if, for
  all $x$ near $[0, 1]^m$, $\|h(x) - f(x, \delta(x), 0)\| < \epsilon$, $|\delta(x)| < \epsilon$ and, for
  every non-zero $\bar u$ in $\Gamma_{d\delta(x)} \times \{0\}$,
  $\big\|(dh(x) \circ p_m - \varphi(x, \delta(x), 0)) \bar u\big\| < \epsilon\|\bar u\|$.

  In particular the zeroth-order part of the definition of the relation $\Rel_{ha}(\sigma, \epsilon)$
  simply expresses that $h$ should be $\epsilon$-close to $x \mapsto f(x, \delta(x), 0)$ and
  $\delta$ should be $\epsilon$-small.

  Let  $f_1$ be an extension of $(x, \delta(x), 0) \mapsto h(x)$ near $A_\delta$.
  In particular $\rst{f_1}{A_\delta} = h \circ \rst{p_m}{A_\delta}$. Differentiating this relation
  gives
  \[
    \rst{df_1(x, \delta(x), 0)}{T_{(x, \delta(x), 0)}A_\delta} = dh(x) \circ \rst{p_m}{T_{(x, \delta(x), 0)}A_\delta}.
  \]
  Since $TA_\delta = \Gamma_{d\delta} \times \{0\}$, for any such extension $f_1$,
  the first order part of $\Rel_{ha}(\sigma, \epsilon)$
  expresses that $df_1$ is $\epsilon$-close to $\varphi$ on each $T_{(x, \delta(x), 0)}A_\delta$.

  This proves that if $(\delta, f_1)$ is a holonomic $\epsilon$-approximation of $\sigma$ near $A$
  then $(\delta, h)$ is a solution of $\Rel_{ha}(\sigma, \epsilon)$.
  Conversely, suppose $(\delta, h)$ is a solution. We want a extension $f_1$
  of $(x, \delta(x), 0) \mapsto h(x)$ near $A_\delta$ such that $(\delta, f_1)$ is a
  solution to the holonomic $\epsilon$-approximation problem for $\sigma$ near $A$.
  Let $\nu$ be the normal bundle of $A_\delta$ for the Euclidean metric on $\mathbb{R}^m \times \mathbb{R} \times \mathbb{R}^k$.
  The Euclidean exponential map $(m, v) \mapsto m + v$ is a diffeomorphism from a
  neighborhood of the zero section in $\nu$ onto a neighborhood of $A_\delta$.
  Using it, we can extend any map defined on $A_\delta$ while prescribing, at each
  point of $A_\delta$, the derivative on $(TA_\delta)^\perp$.
  Specifically, we extend $(x, \delta(x), 0) \mapsto h(x)$ by setting
  $$f_1((x, \delta(x), 0) + v) = h(x) + \varphi(x, \delta(x), 0)v.$$
  We need to prove that $df_1$ and $\varphi$ are $\epsilon$-close near $A_\delta$. Since
  $A_\delta$ is compact and the condition is open, it is enough to prove this at
  every point of $A_\delta$.
  We fix such a point $p$.
  Let $\bar u$ be any non-zero vector tangent to $\mathbb{R}^m \times \mathbb{R} \times \mathbb{R}^k$ at $p$.
  We can decompose it as $\bar u = u_T + v$ with $u_T$ in $T_pA_\delta$ and
  $v$ in $(T_pA_\delta)^\perp$.
  We have $(df_1 - \varphi) \bar u = (df_1 - \varphi) u_T$ since
  $df_1 = \varphi$ on $(T_pA_\delta)^\perp$ by construction.
  If $u_T = 0$ then $\|(df_1 - \varphi) u_T\| = 0 < \epsilon\|\bar u\|$.
  If $u_T \neq 0$ then, since $(\delta, h)$ is a solution of $\Rel_{ha}(\sigma, \epsilon)$,
  $\|(df_1 - \varphi) u_T\| < \epsilon\|u_T\| \leqslant \epsilon\|\bar u\|$
  where the last inequality follows from Pythagoras theorem since
  $v$ is perpendicular to $u_T$.
\end{proof}

\begin{remark}\label{rem:explicite}
  The above proof extends $(x, \delta(x), 0) \mapsto h(x)$ to a function that is linear on
  fibers of the Euclidean tubular neighborhood of $A_\delta$. This is geometrically
  very natural but it does not lead to fully explicit formulas, even when $\delta$
  and $h$ are completely explicit, because the inverse of the normal exponential
  map is not computable in general. However one can write a fully explicit map
  that has the same derivative at each point of $A_\delta$ starting with the ansatz
  $(x, y, z) \mapsto h(x) + (y - \delta(x))g(x) + \varphi(x, y, z) \circ p$, where $p$
  is the projection of $\mathbb{R}^m \times \mathbb{R} \times \mathbb{R}^k$ onto the $\mathbb{R}^k$ factor and $g$ is some
  unknown function. Using that $TA_\delta^\perp = \Span(\nabla\delta - \partial_y) \oplus \mathbb{R}^k$,
  a quick computation of the differential of the previous ansatz reveals that
  \[
    g(x) = \frac1{1 + \|\nabla\delta(x)\|^2}\big(dh(x)\nabla\delta(x) - \varphi(x, \delta(x), 0)(\nabla\delta(x) - \partial_y)\big)
  \]
  is suitable.
\end{remark}

\subsection{Ampleness for holonomic approximation}%
\label{sub:ampleness_for_holonomic_approximation}

This section is devoted to the proof of the second half of the main theorem,
asserting that $\Rel_{ha}(\sigma, \epsilon)$ is open and ample for every $\sigma$ and $\epsilon$.
Openness is clear, so we need to understand the affine geometry of slices of
$\Rel_{ha}(\sigma, \epsilon)$. Note that, in contrast to the previous section,
this section contains no differential calculus, only elementary affine geometry
and bilinear algebra. We fix $\sigma$ and $\epsilon$ and omit them from the
notation $\Rel_{ha}$.

\Cref{lem:ample_coordinates} below is almost purely setting up notations hiding
irrelevant details. But it also features a very simple affine parametrization.
Since affine transformations preserve the ampleness condition, we will be able
to work using only this parametrization.

\begin{lemma}
  \label{lem:ample_coordinates}
  Let $P$ be a principal affine subspace in $J^1(\mathbb{R}^m, \mathbb{R} \times \mathbb{R}^n)$.
  If the slice $\Rel_{ha} \cap P$ is not empty then there is
  a linear form $\lambda$ on $\mathbb{R}^{m-1}$, a linear map $\psi \in \Hom(\mathbb{R}^{m-1}, \mathbb{R}^n)$ and
  an affine isomorphism from $P$ to $\mathbb{R} \times \mathbb{R}^n$ which sends $\Rel_{ha} \cap P$ to
  \[
    \Omega_{\lambda, \psi, \epsilon} = \left\{(a, b) \in \mathbb{R} \times \mathbb{R}^n \;\left|\;
      \begin{aligned}
      &\forall (u, u') \in (\mathbb{R}^{m-1} \times \mathbb{R}) \setminus \{0\},\\
      &\big\|u'b + \psi u\big\|^2 < \epsilon^2\Big((u')^2 + \|u\|^2 + (au' + \lambda u)^2\Big)
      \end{aligned}
    \right.\right\}
  \]
\end{lemma}

\begin{proof}
  By definition of principal affine subspaces,
  $P = \{(x, (y, w), \theta) | \theta_{|\tau} = L\}$ for some $x \in \mathbb{R}^m$, $(y, w) \in \mathbb{R} \times \mathbb{R}^n$,
  some hyperplane $\tau$ in $T_x \mathbb{R}^m = \mathbb{R}^m$ and some
  $L = (Y_0, W_0) \in \Hom(\tau, \mathbb{R} \times \mathbb{R}^n)$.
  Let $v$ be a unit vector orthogonal to $\tau$. Using suitable Euclidean
  coordinates, we can assume $\tau = \mathbb{R}^{m-1} \times \{0\} \subset \mathbb{R}^{m-1} \times \mathbb{R}$ and $v = (0, 1)$.
  In particular, the linear form $\pi$ such that $\ker \pi = \tau$ and $\pi(v) = 1$
  is simply $(u, u') \mapsto u'$.

  As explained after \cref{def:principal}, $v$ and $\pi$ allow to identify $P$
  with the target space $\mathbb{R} \times \mathbb{R}^n$. Namely we can write $(Y,W)$ in $P$ as
  $(u, u') \mapsto (Y_0u + au', W_0u + u'c)$ for some $(a, c)  \in \mathbb{R} \times \mathbb{R}^n$.
  Note also that the graph of $Y$ is the space of vectors
  $(u, u') + Y(u, u') \partial_y$ and that $W \circ p_m$ evaluated on such a vector is
  simply $W_0u + u'c$ since $p_m(\partial_y) = 0$.

  Assume the slice is not empty. In particular $x,y,w$ are fixed by $P$ and satisfy the conditions
  $\|w - f(x,y)\|< \epsilon$ and $|y|< \epsilon$.
  Then, using the above identification, the slice is the set of
  $(a, c) \in \mathbb{R} \times \mathbb{R}^n$ such that, for all non-zero $(u, u')$,
  \[
    \|W_0u + u'c - \varphi(x, y, 0)((u, u') + (Y_0u + au')\partial_y)\| < \epsilon\|(u, u') + (Y_0u + au')\partial_y\|.
  \]
  We square both sides of the equation and use Pythagoras' theorem to rewrite
  the right hand side as $\epsilon^2(\|u\|^2 + (u')^2 + (Y_0u + au')^2)$.
  The promised affine isomorphism sends $(a, c)$ to
  $(a, b)$ where $b = c - a\varphi(x, y, 0)\partial_y - \varphi(x, y, 0)v$.
  We set $\lambda = Y_0$ and define $\psi$ as $u \mapsto W_0u - \varphi(x, y, 0)u - Y_0u\, \varphi(x, y, 0)\partial_y$.
\end{proof}

As explained above, the second part of the main theorem follows from the fact
that each set $\Omega_{\lambda, \psi, \epsilon}$ appearing in the above lemma is ample.
This is the announced purely geometric problem. So we fix a linear form
$\lambda$ on $\mathbb{R}^{m-1}$ and a linear map $\psi$ from $\mathbb{R}^{m-1}$ to $\mathbb{R}^n$.

We note in passing that if $m = 1$ then the situation is already clear. In that
case $\mathbb{R}^{m-1} = \{0\}$ so that $\lambda$ and $\psi$ can only be zero, and checking the
condition for every non-zero $u'$ is equivalent to checking it for $u' = 1$
since the condition is homogeneous.
We then have
$\Omega_{0, 0, \epsilon} = \left\{(a, b) \in \mathbb{R} \times \mathbb{R}^n \;\left|\; \|b\|^2 - \epsilon^2a^2 < \epsilon^2 \right.\right\}$
which is ample.

The next lemma finishes the case where the target has dimension one since it
proves the slice in this case is again the interior of a hyperbola. In this
$n=1$ case, $\lambda$ and $\psi$ play more symmetric roles. This is one reason where we
use the letter $\mu$ instead of $\psi$ in the next lemma. A more serious reason is
the general case will reduced to this lemma applied $n$ times, once for each
component of $\psi$.

\begin{lemma}
  \label{lem:slice_n_eq_1}
  Let $\lambda$ and $\mu$ be linear forms on $\mathbb{R}^{m-1}$, and $\epsilon$ be a positive real number.
  If $\Omega_{\lambda, \mu, \epsilon}$ is not empty then there exist positive real numbers
  $\kappa$ and $\eta$ and a real number $m_0$ such that
  \[
    \Omega_{\lambda, \mu, \epsilon} = \big\{(a, b)\in \mathbb{R} \times \mathbb{R} \;|\; (b - m_0a)^2 - \kappa^2a^2 < \eta^2\big\}.
  \]
\end{lemma}

\begin{proof}
  First note that, for every $(a, b)$,
  $(a,b) \in \Omega_{\lambda, \mu, \epsilon} \Leftrightarrow (a,b/\epsilon)\in \Omega_{\lambda, \mu/\epsilon, 1}$.
  This observation allows to assume $\epsilon = 1$ without loss of generality.

  In this proof we will drop the subscripts in the notation
  $\Omega_{\lambda, \mu, 1}$.
  We denote by $\lambda^\sharp$ and $\mu^\sharp$ the vectors dual to $\lambda$ and $\mu$ for
  the Euclidean structure on $\mathbb{R}^{m-1}$.
  We set
  \[
    A = \Id + \lambda \otimes \lambda^\sharp - \mu \otimes \mu^\sharp \in \End(\mathbb{R}^{m-1}) \quad\text{and}\quad
    B(a, b) = b \mu^\sharp - a \lambda^\sharp \in \mathbb{R}^{m-1}.
  \]
  Note that $A$ is a symmetric endomorphism of $\mathbb{R}^{m-1}$ which does not depend on
  $(a, b)$.

  We assume that $\Omega$ is not empty.
  The condition defining $\Omega$ is homogeneous (of degree 2) in $(u, u')$ hence
  it is true for all $(u, u') \neq 0$ if and only if it is true for all
  $(u, 0)$ with $u \neq 0$ and for all $(u, 1)$.
  The first condition is:
  \[
     \forall u \in \mathbb{R}^{m-1} \setminus \{0\}, \big(\mu u\big)^2 < \|u\|^2 + (\lambda u)^2
  \]
  which means $A$ is positive definite (here we used that $\Omega$ is not empty).
  The second condition is
  \[
     \forall u \in \mathbb{R}^{m-1}, \big(b + \mu u\big)^2 < 1 + \|u\|^2 + (a + \lambda u)^2
  \]
  which can be expanded and, gathering terms by degree in $u$, rewritten as:
  \[
    b^2 - a^2 - 1 < \langle u, Au\rangle - 2\langle B, u\rangle.
  \]
  Since $A$ is positive definite, it is invertible and has a square root (which
  is also symmetric and invertible).
  So we can rewrite the right hand-side as
  $\|A^{1/2}u - A^{-1/2}B\|^2 - \|A^{-1/2}B\|^2$. This is bounded below by
  $-\|A^{-1/2}B\|^2 = -\langle B, A^{-1}B\rangle$, and this bound is attained (when $u = A^{-1}B$). Hence
  \[
    \Omega = \left\{ (a, b) \;|\; b^2 + \langle B(a, b), A^{-1}B(a, b)\rangle - a^2 < 1 \right\}
  \]
  We expand the above equation in powers of $a$ and $b$ to get:
  \[
    (1 + \langle\mu^\sharp, A^{-1}\mu^\sharp\rangle)b^2 -2ab\langle\mu^\sharp, A^{-1}\lambda^\sharp\rangle + (\langle\lambda^\sharp, A^{-1}\lambda^\sharp\rangle - 1)a^2 < 1.
  \]
  We set $N = (1 + \langle\mu^\sharp, A^{-1}\mu^\sharp\rangle)^{1/2}$, which is positive, and rewrite the
  equation as:
  \[
    \left(Nb - \frac{\langle\mu^\sharp, A^{-1}\lambda^\sharp\rangle}{N} a\right)^2 - K a^2 < 1
  \]
  where
  \[
    K = 1 + \frac{\langle\mu^\sharp, A^{-1}\lambda^\sharp\rangle^2}{N^2} - \langle\lambda^\sharp, A^{-1}\lambda^\sharp\rangle.
  \]
  It suffices to prove $K$ is positive and then divide everything by $N^2$ to
  find the announced description.
  We will prove that $K = 1/(1 + \|\lambda\|^2)$, but unfortunately the computation is
  not pleasant.

  Let first explain the degenerate case where $\lambda$ and $\mu$ are linearly dependent.
  If $\lambda = 0$ then the claim is clear.
  Otherwise we can write $\mu = k\lambda$ for some real number $k$, then $A$ simplifies to
  $\Id + (1-k^2)\lambda \otimes \lambda^\sharp$, so $A\lambda^\sharp = (1 + (1 - k^2)\|\lambda\|^2)\lambda^\sharp$, with
  $1 + (1 - k^2)\|\lambda\|^2 > 0$ since $A$ positive. This formula allows to compute
  $A^{-1}\lambda^\sharp$ and then compute $K = 1/(1 + \|\lambda\|^2)$.

  Now assume that $\lambda$ and $\mu$ are linearly independent. Note that, as $A = \Id +
  \lambda \otimes \lambda^\sharp - \mu \otimes \mu^\sharp$, the space $P=Span(\lambda^\sharp,\mu^\sharp)$ is stable. Specifically:
	\begin{equation*}
	A\lambda^\sharp = (1+\|\lambda\|^2)\lambda^\sharp - (\mu\lambda^\sharp)\mu^\sharp, \quad
	A\mu^\sharp = (\mu\lambda^\sharp)\lambda^\sharp + (1-\|\mu\|^2)\mu^\sharp
	\end{equation*}
	and, restricted to $P$, $\det A = (1+\|\lambda\|^2)(1-\|\mu\|^2)+ (\mu\lambda^\sharp)^2$. We now deduce
	\begin{equation*}
	A^{-1}\lambda^\sharp = \frac{1}{\det A}\big((1-\|\mu\|^2)\lambda^\sharp + (\mu\lambda^\sharp)\mu^\sharp \big), \quad
	A^{-1}\mu^\sharp = \frac{1}{\det A}\big( -(\mu\lambda^\sharp)\lambda^\sharp + (1+\|\lambda\|^2)\mu^\sharp \big)
	\end{equation*}
	so
	\begin{align*}
	\lambda A^{-1}\lambda^\sharp &= \frac{1}{\det A}\big((1-\|\mu\|^2)\|\lambda\|^2 + (\mu\lambda^\sharp)^2 \big) = 1 - \frac{(1-\|\mu\|^2)}{\det A},\\
	\mu A^{-1}\mu^\sharp &= \frac{1}{\det A}\big( -(\mu\lambda^\sharp)^2 + (1+\|\lambda\|^2)\|\mu\|^2 \big)= -1 + \frac{(1+\|\lambda\|^2)}{\det A}\\
	\lambda A^{-1}\mu^\sharp &= \frac{\mu\lambda^\sharp}{\det A}
	\end{align*}
  We then obtain again $K = 1/(1+\|\lambda\|^2)$ which is positive.
\end{proof}

We now return to the general case where the target dimension $n$ is any natural number.

\begin{lemma}
  \label{lem:connected_slice}
  Each $\Omega_{\lambda, \psi, \epsilon}$ is either empty or connected.
\end{lemma}

\begin{proof}
  We fix $(\lambda, \psi, \epsilon)$ and set $\Omega = \Omega_{\lambda, \psi, \epsilon}$. We assume $\Omega$ is not empty
  and will prove it is star-shaped with respect to the origin.

  In the definition of $\Omega$,
  we can specialize to $u' = 0$ to get that, for every $(a, b)$ in $\Omega$:
  \begin{equation}\label{eq:origin}
    \forall u \in \mathbb{R}^{m-1} \setminus \{0\},\; \|\psi u\|^2 < \epsilon^2(\|u\|^2 + (\lambda u)^2).
  \end{equation}
  This condition \eqref{eq:origin} does not depend on $(a, b)$. Since
  $\Omega$ is not empty, we learn that~\eqref{eq:origin} holds.

  We now prove that the origin is in $\Omega$. Fix some $(u, u') \in (\mathbb{R}^{m-1} \times \mathbb{R}) \setminus
  \{0\}$. If $u \neq 0$ then $\|\psi u\|^2 < \epsilon^2\big((u')^2 + \|u\|^2 + (\lambda u)^2\big)$ thanks
  to condition \eqref{eq:origin} and $(u')^2 \geqslant 0$. Otherwise $u' \neq 0$ and the
  condition to check reduces to $0 < \epsilon^2(u')^2$.

  Next, assuming $(a, b)$ is in $\Omega$ and $t$ is in $(0, 1]$, we need to prove
  that $t(a, b)$ is in $\Omega$. We fix some non-zero $(u, u')$ and compute
  \begin{align*}
    \big\|u'tb + \psi u\big\|^2 &= t^2\big\|u'b + \psi(u/t)\big\|^2 \\
                           &< t^2\epsilon^2\Big((u')^2 + \|u/t\|^2 + (au' + \lambda(u/t))^2\Big) \text{ since $(a, b) \in \Omega$} \\
                           &= \epsilon^2\Big(t^2(u')^2 + \|u\|^2 + (tau' + \lambda u)^2\Big) \\
                           &\leqslant \epsilon^2\Big((u')^2 + \|u\|^2 + (tau' + \lambda u)^2\Big) \text{ since $t \leqslant 1$.}
  \end{align*}
  Hence $(ta, tb)$ is in $\Omega$.
\end{proof}

As explained earlier, the next lemma finishes the proof of the main theorem.

\begin{lemma}
  \label{lem:ample_slice}
  Each $\Omega_{\lambda, \psi, \epsilon}$ is ample.
\end{lemma}

\begin{proof}
Since the empty set is ample, we can assume $\Omega = \Omega_{\lambda, \psi, \epsilon}$ is not empty.
Since \Cref{lem:connected_slice} guarantees
that $\Omega$ is connected, it suffices to prove that $\Omega$ contains a non-empty ample
set. Using the definition of the sup norm on $\mathbb{R}^n$, we get:
\[
  \Omega = \left\{(a, b) \;|\; \forall j,\, (a, b_j) \in \Omega_{\lambda, \psi_j, \epsilon} \right\}
\]
where $\psi_j$ is the composition of $\psi$ and the projection onto the
$j$-th factor of $\mathbb{R}^n$.
Since we assumed $\Omega$ is non-empty, each $\Omega_{\lambda, \psi_j, \epsilon}$ is non-empty.
So \Cref{lem:slice_n_eq_1} gives us positive numbers $\kappa_j$ and $\eta_j$ and some
numbers $m_{0, j}$ such that
$\Omega_{\lambda, \psi_j, \epsilon} = \{(a, b_j) \;|\; (b_j - am_{0, j})^2 - \kappa_j^2 a^2 < \eta_j^2\}$.
We denote by $m_0$ the vector with components $m_{0, j}$.
We set $\kappa = \min_j(\kappa_j)$ and $\eta = \min_j(\eta_j)$ so that
\[
  \left\{(a, b) \;|\; \|b - am_0\|^2 - \kappa^2a^2 < \eta^2 \right\} \subset \Omega
\]
The set on the left-hand side is non-empty and ample hence the proof is completed.
\end{proof}

\emergencystretch=1.5em
\hbadness=2000
\printbibliography

\addvspace{2cm}
{\footnotesize%
\noindent
P.~Massot: \textsc{Laboratoire de Mathématiques d’Orsay, CNRS, Université Paris-Saclay.}

\noindent
\texttt{patrick.massot@math.cnrs.fr}

\bigskip

\noindent
M.~Theillière: \textsc{Université du Luxembourg.}

\noindent
\texttt{melanie.theilliere@uni.lu}
}
\end{document}